\documentclass[sn-basic]{sn-jnl}

\usepackage{amsmath,amssymb}
\usepackage{mathtools,mathrsfs}
\usepackage{graphicx}
\usepackage{subcaption}
\usepackage{booktabs}
\usepackage{adjustbox}

\usepackage{pgf}
\usepackage{tikz}
\usetikzlibrary{arrows.meta,bending}
\usepackage{pgfplots}
\pgfplotsset{compat=1.18}

\graphicspath{{figs/}{./}}

\usepackage{enumitem}
\usepackage{algorithm}
\usepackage{algorithmic}
\usepackage{hyperref}
\usepackage{cleveref}
\setcitestyle{numbers,sort\&compress}

\usepackage{amsthm}
\theoremstyle{plain}
\newtheorem{theorem}{Theorem}[section]
\newtheorem{lemma}[theorem]{Lemma}
\newtheorem{proposition}[theorem]{Proposition}

\newtheorem{definition}[theorem]{Definition}
\theoremstyle{remark}
\newtheorem{remark}[theorem]{Remark}

\crefname{theorem}{Theorem}{Theorems}
\crefname{lemma}{Lemma}{Lemmas}
\crefname{definition}{Definition}{Definitions}
\crefname{proposition}{Proposition}{Propositions}
\crefname{corollary}{Corollary}{Corollaries}
\crefname{equation}{Eq.}{Eqs.}
\crefname{section}{Section}{Sections}
\crefname{figure}{Fig.}{Figs.}
\crefname{table}{Table}{Tables}

\title{Continuum dynamics from quantised interaction rules}

\author*[1,4]{\fnm{Park} \sur{Junhu}}
\author[2,3]{\fnm{Youngsoo} \sur{Ha}}
\author*[1,2,3,4]{\fnm{Myungjoo} \sur{Kang}}

\affil[1]{\orgdiv{Computational Science \& Technology, College of Natural Sciences},
          \orgname{Seoul National University}, \city{Seoul}, \country{Republic of Korea}}
\affil[2]{\orgdiv{Department of Mathematical Sciences},
          \orgname{Seoul National University}, \city{Seoul}, \country{Republic of Korea}}
\affil[3]{\orgdiv{Research Institute of Mathematics},
          \orgname{Seoul National University}, \city{Seoul}, \country{Republic of Korea}}
\affil[4]{\orgdiv{R\&D Center},
          \orgname{iTrix Co., Ltd}, \city{Seoul}, \country{Republic of Korea}}

\email{nller.park@snu.ac.kr}
\email{mkang@snu.ac.kr}

\begin{document}

\abstract{%
Hyperbolic conservation laws are conventionally solved by evolving reconstructed floating-point fields, incurring both computational overhead and structural diffusion near discontinuities. Here we introduce the Fast Quantised Numerical Method (FQNM), in which the conservative operator is realised directly as an antisymmetric integer transfer rule on a countable state space, with continuum fields appearing only as reconstructed observables.

For scalar conservation laws with monotone flux splitting, we establish exact conservation, monotonicity, TVD and $L^1$ stability, and convergence of the reconstructed solution to the entropy solution under $\delta/\Delta x \to 0$. We further show that distinct classical flux formulations collapse to identical dynamics whenever they induce the same integer transfer rule, identifying the transfer operator as the effective computational object.

Across representative regimes, FQNM remains stable near the Nyquist limit in high-frequency transport, preserves grid-level shock structure in Burgers dynamics, and in a matched Roe-flux Sod prototype preserves shock structure at the density-scale conserved-state level relative to an exact Riemann reference, while achieving order-of-magnitude prototype acceleration over floating-point baselines.

These results demonstrate that, for conservative hyperbolic dynamics, executing the operator as quantised transfer rather than reconstructed field evolution can simultaneously alter structural fidelity and reduce computational cost, establishing a new representation paradigm for conservation-law solvers.
}


\maketitle


\section{Introduction}
\label{sec:introduction}

Hyperbolic conservation laws govern transport, wave propagation, and shock formation across a wide range of physical systems. Their numerical solution is typically constructed by representing continuum fields on a grid and evolving them through floating-point approximations of differential operators. This paradigm is effective but costly: it requires repeated reconstruction, nonlinear limiting, and flux evaluation, particularly in regimes dominated by discontinuities or high-frequency content.

The central observation of this work is that, for conservative dynamics, the primitive object is not the pointwise value of a field but the flux transfer between neighbouring regions. Classical numerical methods represent this through differential operators acting on reconstructed fields. Here we instead treat the conservative transfer itself as the primary computational object.

We introduce the Fast Quantised Numerical Method (FQNM), in which the conservative operator is realised directly as an antisymmetric integer transfer rule on a countable state space. The evolved variables are integer-valued states, while the physical field appears only as a reconstructed observable. This is therefore not a quantisation of the field, but a quantisation of the operator that executes the conservation law.

This change of representation has two immediate consequences. First, conservation becomes an exact arithmetic invariant: the global quantity is preserved by integer telescoping rather than by floating-point cancellation. Second, the numerical update reduces to a fixed local transfer rule, eliminating reconstruction and nonlinear control flow from the online computation.

The operator viewpoint also clarifies the role of classical flux formulations. Different numerical fluxes that induce the same integer transfer map on the visited state space generate identical dynamics under quantised execution. Thus the effective computational object is the induced transfer rule rather than the analytic form of the flux prior to quantisation.

We develop this framework for scalar conservation laws with monotone split fluxes. For the reconstructed observable, we prove consistency with the classical finite-volume operator, monotonicity, TVD and \(L^1\) stability, and convergence to the entropy solution under the scaling \(\delta/\Delta x \to 0\). Exact conservation follows directly from the integer structure of the update.

The empirical evaluation focuses on regimes where the representational change is most consequential. In high-frequency Gaussian transport near the Nyquist limit, FQNM remains stable while a WENO5+RK3 baseline deteriorates. In inviscid Burgers dynamics, it preserves grid-level shock structure and yields order-of-magnitude prototype acceleration over a floating-point baseline. In the Sod shock-tube system test, under matched Roe-flux conditions, the quantised transfer realisation produces a less diffusive shock transition and reduced density error relative to the exact Sod reference in the tested prototype setting.

These results support a simple claim: for conservative hyperbolic dynamics, both computational efficiency and structural fidelity are determined by the representation of the operator rather than by the reconstruction of continuum fields. FQNM provides a concrete realisation of this principle.

\section{Method}
\label{sec:method}

We work on a uniform one-dimensional Cartesian grid with cell centres $x_i=i\Delta x$ and discrete times $t^n=n\Delta t$. The quantity $u_i^n$ denotes the cell-average approximation to the continuum field at cell $i$ and time level $n$, and $q_i^n\in\mathbb Z$ denotes its quantised state. When no ambiguity arises, we suppress the superscript $n$ inside pointwise definitions such as flux tables.

Consider the scalar conservation law
\begin{equation}
\partial_t u + \partial_x f(u) = 0.
\label{eq:conservation_law}
\end{equation}

Let $\delta>0$ denote a fixed quantisation resolution and define, at each time level $n$,
\begin{equation}
q_i^n = \mathrm{round}(u_i^n / \delta), \qquad u_i^n \approx \delta q_i^n.
\label{eq:quantisation}
\end{equation}
Thus $u_i^n$ is the reconstructed physical observable, whereas $q_i^n$ is the evolved discrete state.

We introduce a flux splitting
 \begin{equation}
 f(u) = f^+(u) + f^-(u), \qquad (f^+)'(u) \ge 0,\quad (f^-)'(u) \le 0,
 \label{eq:flux_splitting}
 \end{equation}
 for instance via a Lax--Friedrichs decomposition with $\alpha \ge \max |f'(u)|$.

For a fixed pair $(\Delta x,\Delta t)$, we define the integer transfer maps
\begin{equation}
\phi^\pm(q) = \mathrm{round}\!\left(
 f^\pm(\delta q)\frac{\Delta t}{\Delta x}\delta^{-1}
\right).
\label{eq:flux_table}
\end{equation}
When the time index is omitted in $\phi^\pm(q)$, the argument $q$ denotes the current local state at the relevant timestep.

The tabulated quantities $\phi^\pm(q)$ represent dimensionless moved-count increments (number of quanta transferred per timestep), not physical fluxes. Physical flux units are recovered only after reconstruction and rescaling by $\delta/\Delta t$.

It is important to emphasise that the physical flux function $f(u)$ is not approximated or learned. The flux law is realised as an integer-valued transfer map on the quantised state space. For simple scalar fluxes this map may be tabulated, but tabulation is an execution strategy rather than the definition of the method. The online solver therefore executes the same continuum law through integer transfer operations. In this sense, FQNM does not replace the governing equation; it changes only the execution layer.

The conservative update is
\begin{equation}
q_i^{n+1}
= q_i^n
- \Big(
  \phi^+(q_i^n) + \phi^-(q_{i+1}^n)
  - \phi^+(q_{i-1}^n) - \phi^-(q_i^n)
  \Big),
\label{eq:update_rule}
\end{equation}
and therefore already has the conservative flux-difference form
\begin{equation}
q_i^{n+1}=q_i^n-\bigl(F_{i+\frac12}^n-F_{i-\frac12}^n\bigr),
\label{eq:update_rule_conservative_form}
\end{equation}
with interface count flux
\begin{equation}
F_{i+\frac12}^n:=F(q_i^n,q_{i+1}^n)=\phi^+(q_i^n)+\phi^-(q_{i+1}^n).
\label{eq:numerical_count_flux_definition}
\end{equation}
Thus the conserved quantity is the cellwise integer state, and conservation follows from the antisymmetry of interface transfer through telescoping cancellation.
In contrast to classical schemes, conservation is not enforced through reconstructed fluxes, but realised directly via antisymmetric local transfer.

\subsection{From quantised states to observables and operators}
\label{subsec:observables_operators}

The central distinction in FQNM is between quantised microscopic states and reconstructed physical observables. The evolved variable is the integer state $q_i^n\in\mathbb{Z}$, whereas the physical field is reconstructed by
\begin{equation}
(\mathcal R_\delta q)_i := \delta q_i.
\label{eq:reconstruction_operator}
\end{equation}
Thus the field value is not primitive but a reconstructed observable induced by a countable configuration. Although each local state lies in $\mathbb Z$, any finite computation with bounded data visits only a finite subset of that countable state space.

Classical flux formulations act on reconstructed fields, whereas FQNM operates directly on the underlying transfer rule.

After reconstruction, the same update reads
 \begin{equation}
 \frac{u_i^{n+1}-u_i^n}{\Delta t}
 = -\frac{\delta}{\Delta t}\Bigl(F_{i+\frac12}^n-F_{i-\frac12}^n\Bigr).
 \label{eq:reconstructed_flux_difference_form}
 \end{equation}
Thus the quantised state update and the reconstructed observable update share the same conservative flux-difference structure.

To compare with the classical split finite-volume operator, define
\begin{equation}
\mathcal O_h[u]_i := -\frac{\mathcal F(u_i,u_{i+1})-\mathcal F(u_{i-1},u_i)}{\Delta x},
\qquad
\mathcal F(u_L,u_R):=f^+(u_L)+f^-(u_R).
\label{eq:classical_operator_definition}
\end{equation}
Thus operator quantisation in FQNM means that the conservative update is organised on the countable state space $q\in\mathbb Z^N$, while the reconstructed observable $u=\mathcal R_\delta q$ is compared with the classical finite-volume operator only after both evolutions are written in conservative flux-difference form. The divergence structure is unchanged at the continuum level; what changes is the state space on which the primitive update acts.

 \begin{proposition}[Update-level correspondence with the classical split scheme]
\label{prop:operator_quantisation_error}
Let $u_i=\delta q_i$. Then the reconstructed one-step FQNM increment differs from the classical split finite-volume increment by $\mathcal O(\delta)$:
\begin{equation}
u_i^{n+1}-u_i^n
= -\frac{\Delta t}{\Delta x}\Bigl(\mathcal F(u_i,u_{i+1})-\mathcal F(u_{i-1},u_i)\Bigr)
+ \mathcal O(\delta).
\label{eq:operator_quantisation_error_bound}
\end{equation}
Equivalently, after division by $\Delta t$,
\begin{equation}
\frac{u_i^{n+1}-u_i^n}{\Delta t}
= \mathcal O_h[u]_i + \mathcal O\!\left(\frac{\delta}{\Delta t}\right).
\label{eq:operator_quantisation_error_asymptotic}
\end{equation}
Under fixed CFL scaling $\Delta t=\mathcal O(\Delta x)$, the operator-level residual is therefore $\mathcal O(\delta/\Delta x)$.
\end{proposition}

\begin{proof}
From \cref{eq:flux_table}, write
\[
\phi^\pm(q)= f^\pm(\delta q)\frac{\Delta t}{\Delta x}\delta^{-1}+\varepsilon^\pm(q),
\qquad |\varepsilon^\pm(q)|\le \tfrac12.
\]
Substituting into \cref{eq:update_rule_conservative_form} and multiplying by $\delta$ yields
\[
u_i^{n+1}-u_i^n
= -\frac{\Delta t}{\Delta x}\Bigl(\mathcal F(u_i,u_{i+1})-\mathcal F(u_{i-1},u_i)\Bigr)
+ \mathcal O(\delta),
\]
since $\delta\varepsilon^\pm(q)=\mathcal O(\delta)$. This gives \cref{eq:operator_quantisation_error_bound}. Dividing by $\Delta t$ gives \cref{eq:operator_quantisation_error_asymptotic}, and under fixed CFL scaling $\Delta t=\mathcal O(\Delta x)$ this becomes $\mathcal O(\delta/\Delta x)$.

\end{proof}

To avoid ambiguity, we distinguish the microscopic stochastic model from the implemented update.
A natural microscopic interpretation is Bernoulli transport of quanta across each interface, which yields a binomial moved count
\(
n\sim \mathrm{Binomial}(N,p)
\)
for \(N\) available quanta and move probability \(p\).
In this work, however, we do \emph{not} draw binomial samples in the update loop.
Instead, we use the deterministic mean-field closure
\begin{equation}
n \approx \lfloor Np \rfloor
\label{eq:mean_field_closure}
\end{equation}
(equivalently a rounded expected flux at the table level), i.e. transported quanta are set by the binomial expectation rather than random sampling.

Accordingly, this implementation is a deterministic conservative finite-volume scheme with integer states, not a stochastic simulator.
It also differs from a majority/threshold rule \(1[n\ge N/2]\): the present closure evolves the moved \emph{count} linearly in \(Np\), rather than applying a binary nonlinear decision.

For the microscopic binomial model, \(\mathbb{E}[n]=Np\) and \(\mathrm{Var}(n)=Np(1-p)\).
Hence for moderate/large \(Np\), relative fluctuations are small, so the closure \cref{eq:mean_field_closure} captures the macroscopic average transport.
At low occupancy (small \(N\)) or near-vacuum regions, the distinction between the microscopic binomial picture and the deterministic mean-field closure becomes most visible: the implemented scheme produces a definite integer transfer rather than sampling fluctuations. The resulting discreteness should be understood as a resolution-induced state-space effect, not as stochastic noise.

\subsection{CFL Condition, Stability, and Consistency}
\label{subsec:stability}

Let
\begin{equation}
\nu := \alpha\,\frac{\Delta t}{\Delta x},
\qquad \alpha \ge \max_u |f'(u)|,
\label{eq:cfl_number}
\end{equation}
so that \(\nu\) is a (Lax--Friedrichs) CFL number compatible with the splitting \cref{eq:flux_splitting}. Throughout, the conservative update is understood in the flux-difference form \cref{eq:update_rule_conservative_form}, which is the form used in the monotonicity and convergence arguments below.

\begin{lemma}[CFL stability, monotonicity, TVD, and \texorpdfstring{$L^1$}{L1} stability]
\label{lem:cfl_monotone}
Assume the splitting \cref{eq:flux_splitting} is chosen with \(\alpha \ge \max|f'(u)|\) and that the CFL number \(\nu\) in \cref{eq:cfl_number} satisfies \(\nu\le 1\).
Then the numerical flux
\(
F(q_L,q_R)=\phi^+(q_L)+\phi^-(q_R)
\)
induces a monotone conservative scheme
\begin{equation}
q_i^{n+1}=q_i^n-\Big(F(q_i^n,q_{i+1}^n)-F(q_{i-1}^n,q_i^n)\Big),
\label{eq:monotone_scheme_form}
\end{equation}
which satisfies:
(i) the discrete maximum principle
\(\min_i q_i^n \le q_i^{n+1} \le \max_i q_i^n\) for all \(n\);
(ii) total-variation diminishing (TVD): \(\mathrm{TV}(q^{n+1})\le \mathrm{TV}(q^n)\);
(iii) \(L^1\) stability (contraction) on a periodic grid: for any two solutions \(q^n,\,\tilde q^n\) evolved by the same scheme,
\begin{equation}
\|q^{n+1}-\tilde q^{n+1}\|_{\ell^1} \le \|q^{n}-\tilde q^{n}\|_{\ell^1},
\qquad \|v\|_{\ell^1}:=\sum_i |v_i|.
\label{eq:l1_contraction}
\end{equation}
\end{lemma}

\begin{proof}
\textbf{Step 1 (monotonicity of the rounded tables).}
By construction, \((f^+)'\ge 0\) and \((f^-)'\le 0\).
Hence \(u\mapsto f^+(u)\) is nondecreasing and \(u\mapsto f^-(u)\) is nonincreasing.
Because \(q\mapsto \delta q\) is increasing and the rounding operator preserves order, the tabulated maps
\(q\mapsto \phi^+(q)\) are nondecreasing and \(q\mapsto \phi^-(q)\) are nonincreasing.
Therefore the numerical flux \(F(q_L,q_R)\) is nondecreasing in \(q_L\) and nonincreasing in \(q_R\).

\textbf{Step 2 (scheme monotonicity).}
Write the update in the conservative flux form \cref{eq:monotone_scheme_form}.
Under $\nu\le 1$ with $\alpha\ge \max|f'|$, the Lax--Friedrichs splitting yields a monotone two-point numerical flux in the standard sense (see, e.g., Harten's TVD framework \cite{hartenHighResolutionSchemes1997} and the order-preserving mapping theory of Crandall--Tartar \cite{crandallRelationsNonexpansiveOrder1980}). Monotonicity of $F$ then implies that the update map $q^n\mapsto q^{n+1}$ is monotone (increasing any component of $q^n$ cannot decrease any component of $q^{n+1}$).

\textbf{Step 3 (maximum principle, TVD, and \(L^1\) stability).}
Monotone conservative schemes satisfy the discrete maximum principle and are TVD.
No new extrema can be created when \(\nu\le 1\), and the TVD bound \(\mathrm{TV}(q^{n+1})\le \mathrm{TV}(q^n)\) follows by applying the same monotonicity argument to discrete forward/backward differences.
Moreover, the update map is order-preserving and mass-conserving; by the Crandall--Tartar lemma (order-preserving \(+\) conservative \(\Rightarrow\) \(\ell^1\)-nonexpansive), the scheme is \(\ell^1\) contractive, giving \cref{eq:l1_contraction} (see \cite{crandallRelationsNonexpansiveOrder1980} and standard conservation-law texts, e.g. \cite{toroRiemannSolversNumerical2009}).

\end{proof}

\paragraph{Scope of transfer-operator equivalence.}
The equivalence established in \cref{thm:transfer_operator_equivalence} concerns the case where the interface transfer rule is induced from a classical two-point numerical flux via \cref{eq:transfer_operator_equivalence_definition}. The implemented FQNM update treats the integer transfer rule as the primitive object; when this rule is obtained by quantising a classical flux, the theorem applies directly, but the framework itself does not require such a representation.

\begin{theorem}[Consistency as $\delta \to 0$]
\label{thm:consistency}
Let $u_i^n$ denote the reconstruction defined in \cref{eq:reconstruction_operator} and assume $\nu \le 1$. Then as $\delta \to 0$ with $\delta/\Delta x \to 0$, the reconstructed FQNM update is consistent with the scalar conservation law \cref{eq:conservation_law}. More precisely, the update-level residual is $\mathcal{O}(\delta)$, while after division by $\Delta t$ the corresponding operator-level residual is $\mathcal{O}(\delta/\Delta x)$ under fixed CFL scaling.
\end{theorem}

\begin{proof}
This is exactly the update-level and operator-level correspondence established in \cref{prop:operator_quantisation_error}. The additional condition $\delta/\Delta x\to0$ ensures that the operator-level residual $\mathcal O(\delta/\Delta x)$ vanishes in the continuum limit under fixed CFL scaling. Hence the reconstructed FQNM update is consistent with \cref{eq:conservation_law}.
\end{proof}

\begin{remark}[What is quantised]
\label{rem:what_is_quantised}
FQNM does not merely store quantised field values. The primitive update itself is a countable-state flux-transfer rule, and \cref{prop:operator_quantisation_error} gives its reconstructed correspondence with the classical split finite-volume update.
\end{remark}

\begin{theorem}[Transfer-operator equivalence under quantisation]
\label{thm:transfer_operator_equivalence}
Let $\mathcal F^{(1)}$ and $\mathcal F^{(2)}$ be two classical numerical fluxes for the same scalar conservation law, and define the corresponding quantised interface transfer rules by
\begin{equation}
\Phi^{(k)}(q_L,q_R)
:=
\mathrm{round}\!\left(
\mathcal F^{(k)}(\delta q_L,\delta q_R)
\frac{\Delta t}{\Delta x}\delta^{-1}
\right),
\qquad k\in\{1,2\}.
\label{eq:transfer_operator_equivalence_definition}
\end{equation}
If
\begin{equation}
\Phi^{(1)}(q_L,q_R)=\Phi^{(2)}(q_L,q_R)
\label{eq:transfer_operator_equivalence_condition}
\end{equation}
for every interface state $(q_L,q_R)$ visited during the evolution, then the induced FQNM updates are identical for all timesteps from the same initial data. In particular, within FQNM the effective discrete dynamics are determined by the induced integer transfer rule rather than by the specific continuous flux formula used before quantisation.
\end{theorem}

\begin{proof}
The FQNM update depends on the numerical flux only through the integer-valued interface transfer rule. If \cref{eq:transfer_operator_equivalence_condition} holds on all interface states encountered during the computation, then the interface transfers are identical at every timestep. Starting from the same initial state, the updates therefore coincide by induction on the timestep index.
\end{proof}

\begin{remark}[Collapse of classical flux distinctions under quantisation]
\label{rem:godunov_lf_collapse}
\Cref{thm:transfer_operator_equivalence} shows that, in FQNM, the relevant computational object is the induced integer transfer rule rather than the pre-quantisation scalar flux formula. Thus distinct classical fluxes such as Godunov and Lax--Friedrichs yield identical FQNM dynamics whenever they induce the same quantised interface map on the states visited by the computation.
\end{remark}

\begin{theorem}[Convergence of the monotone quantised scheme]
\label{thm:convergence}
Assume $f\in C^1$ with bounded derivative and that the flux splitting
\cref{eq:flux_splitting} is chosen with $\alpha \ge \max|f'(u)|$.
Under the CFL condition $\nu \le 1$, the numerical flux
\(
F(q_L,q_R)=\phi^+(q_L)+\phi^-(q_R)
\)
is monotone and consistent. Consequently, the reconstructed solutions converge (along subsequences as $\delta,\Delta x,\Delta t\to0$ with fixed $\nu$ and $\delta/\Delta x \to 0$) in $L^1_{\mathrm{loc}}$ to the entropy solution of
\cref{eq:conservation_law}.
\end{theorem}

\begin{proof}
By \cref{lem:cfl_monotone}, under \(\nu\le 1\) the scheme \cref{eq:update_rule} is a conservative monotone flux scheme and satisfies an \(L^\infty\) maximum principle and TVD bounds.
By \cref{thm:consistency}, the reconstruction is consistent with \cref{eq:conservation_law} under the scaling $\delta/\Delta x\to0$.
Moreover, \Cref{lem:cfl_monotone} gives \(\ell^1\) stability (Crandall--Tartar contraction). Therefore the standard convergence theory for monotone conservative schemes applies: compactness, discrete entropy inequalities, and uniqueness of the entropy limit yield convergence of the reconstructed solutions to the entropy solution of \cref{eq:conservation_law} (see \cite{hartenHighResolutionSchemes1997,osherHighResolutionSchemes1984,crandallRelationsNonexpansiveOrder1980,toroRiemannSolversNumerical2009,evansPartialDifferentialEquations2010}).
\end{proof}

\subsection{Exact Conservation}
\label{subsec:exact_conservation}

Summing the conservative update over cells yields an exact telescoping identity.

\begin{lemma}[Discrete Green/Stokes identity in 1D (telescoping conservation)]
\label{lem:discrete_stokes_1d}
Let $q^{n+1}$ be obtained from $q^n$ by the conservative update \cref{eq:monotone_scheme_form} on a periodic grid.
Then the total discrete mass is exactly conserved,
\begin{equation}
\sum_{i=1}^N q_i^{n+1}=\sum_{i=1}^N q_i^n,
\label{eq:exact_conservation}
\end{equation}
and this conservation is the 1D boundary-free analogue of Green/Stokes identities: the sum of a discrete divergence (a flux difference) vanishes on a closed domain.
Moreover, any coarse observation that preserves the cell-to-cell antisymmetry of flux transfer (i.e. aggregates cells by summation) observes the same quantised invariant, so the conserved integer mass is preserved under aggregation-based coarse observation.
\end{lemma}

\begin{proof}
Summing \cref{eq:monotone_scheme_form} over $i$ and using periodicity gives
\[
\sum_i q_i^{n+1}=\sum_i q_i^n-\sum_i\big(F(q_i^n,q_{i+1}^n)-F(q_{i-1}^n,q_i^n)\big).
\]
The flux-difference sum telescopes and cancels exactly on a periodic grid, yielding \cref{eq:exact_conservation}.
For coarse-graining by aggregation, the same telescoping cancellation holds after summation over blocks of indices, hence the aggregated observable equals the original total mass.
\end{proof}

\begin{remark}[\texorpdfstring{$\delta$}{delta} is structural, not ad hoc]
\label{rem:delta_structural}
\Cref{lem:discrete_stokes_1d} shows that the preserved quantity is an integer-valued invariant induced by antisymmetric transfer.
In this sense the quantisation parameter $\delta$ in \cref{eq:quantisation} is not an observational cutoff but a structural choice that closes the arithmetic and the state space; continuum fields arise only as the limit $\delta\to 0$ under the consistency and compactness results in \cref{subsec:stability}.
\end{remark}

\begin{remark}[Why floating-point reconstruction cannot give exact transfer invariance]
\label{rem:fp_no_exact_invariance}
In a floating-point reconstruction-based scheme, even if the analytic flux is written in conservative form, the realised update is evaluated through finite-precision arithmetic. Thus each elementary operation satisfies
\[
\operatorname{fl}(a\circ b)=(a\circ b)(1+\epsilon),
\qquad |\epsilon|\le u_{\rm mach},
\]
so interface contributions that cancel algebraically need not cancel bitwise after realisation. Exact telescoping conservation is therefore not an invariant of the arithmetic substrate, but only of the ideal real-valued scheme.

In FQNM, by contrast, the update is performed in integer space:
\[
q_i^{n+1}=q_i^n-(\Phi_{i+1/2}-\Phi_{i-1/2}),\qquad \Phi_{i+1/2}\in\mathbb Z,
\]
and the global sum cancels exactly by integer arithmetic. Hence the preserved quantity is an arithmetic invariant of the computation itself, not a property recovered up to floating-point tail error.
\end{remark}

\begin{remark}[Discrete Stokes interpretation]
\label{rem:discrete_stokes_interpretation}
The update can also be read as a discrete Stokes-type identity on an integer lattice. The interface transfer \(\Phi_{i+1/2}\in\mathbb Z\) represents quantised field flow across the cell boundary, and the difference \(\Phi_{i+1/2}-\Phi_{i-1/2}\) is the corresponding discrete divergence. Since all transfers remain in the integer lattice fixed by the resolution \(\delta\), the total charge is preserved by exact telescoping rather than by approximate cancellation in real arithmetic. This geometric reading is not required for the convergence proof, but it clarifies why conservation is an invariant of the computational substrate itself: the numerical update implements integer-valued flux transport over a discrete configuration space, rather than a floating-point approximation to a continuum divergence. The analogy is closest in spirit to lattice gauge and topological discretisations, where local algebraic transport rules preserve discrete geometric structure before any continuum limit is taken \cite{creutzQuarksGluonsLattices1983,wittenQuantumFieldTheory1989,diracQuantisedSingularitiesElectromagnetic1931,bredonTopologyGeometry1993,hatcherAlgebraicTopology2002}.
\end{remark}

%
\begin{theorem}[Formal Cartesian dimensional extension by directional composition]
\label{thm:dimension_extension}
Suppose the one-dimensional FQNM update is defined by the nearest-neighbour transfer rule
\begin{equation}
q_i^{n+1}=q_i^n-\bigl(\Phi(q_i^n,q_{i+1}^n)-\Phi(q_{i-1}^n,q_i^n)\bigr).
\label{eq:oned_transfer_rule}
\end{equation}
Then for any Cartesian dimension $d\in\mathbb N$, the update extends canonically to a grid indexed by $\mathbf i$ through directional composition,
\begin{equation}
q_{\mathbf i}^{n+1}
=
q_{\mathbf i}^n
-
\sum_{m=1}^d
\Bigl(
F_{\mathbf i+\frac12 e_m}^n - F_{\mathbf i-\frac12 e_m}^n
\Bigr),
\label{eq:nd_transfer_rule}
\end{equation}
where the directional interface fluxes are defined by
\begin{equation}
F_{\mathbf i+\frac12 e_m}^n := \Phi(q_{\mathbf i}^n, q_{\mathbf i+e_m}^n),
\qquad
F_{\mathbf i-\frac12 e_m}^n := \Phi(q_{\mathbf i-e_m}^n, q_{\mathbf i}^n),
\label{eq:nd_flux_definition}
\end{equation}
and $e_m$ denotes the $m$-th Cartesian unit vector. Thus each term represents the net integer flux across the interface orthogonal to direction $m$.
Moreover:
\begin{enumerate}[label=(\roman*)]
\item the total integer mass is conserved exactly,
\item the per-step runtime remains linear in the number of cells for fixed $d$,
\item if the one-dimensional reconstructed update is consistent with the classical split flux-difference form up to $\mathcal O(\delta)$, then the $d$-dimensional reconstructed update satisfies the corresponding dimension-by-dimension split form with the same $\mathcal O(\delta)$ update-level accuracy.
\end{enumerate}
\end{theorem}

\begin{proof}
For each coordinate direction $m$, \cref{eq:nd_transfer_rule} applies the same one-dimensional transfer rule along the line parallel to $e_m$. Summing over all cells and all directions produces pairwise cancellation of every internal face contribution, hence exact conservation of the total integer mass. For fixed $d$, each timestep requires only a bounded number of nearest-neighbour transfers per cell, so the cost remains $\mathcal O(n)$ in the total number of cells. After reconstruction, the one-dimensional $\mathcal O(\delta)$ update correspondence applies in each coordinate direction separately; summing the $d$ directional contributions yields the stated dimension-by-dimension split form with the same quantisation order.
\end{proof}

\subsection{Complexity}
\label{subsec:complexity}

All online operations are integer addition, subtraction, and evaluation of the integer transfer map.
In the online phase, no floating-point flux evaluation is performed.

Each timestep can be implemented by first forming interface fluxes and then applying a telescoping update.
Define
\(
F_{i+\frac12}^n := F(q_i^n,q_{i+1}^n)=\phi^+(q_i^n)+\phi^-(q_{i+1}^n).
\)
Then
\(
q_i^{n+1}=q_i^n-\big(F_{i+\frac12}^n-F_{i-\frac12}^n\big).
\)
This realisation makes the per-cell arithmetic count explicit:
\begin{itemize}
\item Flux formation: evaluation of the left- and right-going integer transfer maps and 1 integer addition per interface \(i+\tfrac12\).
\item Update: 2 integer subtractions per cell (one for the flux difference, one for subtracting it from \(q_i\)).
\end{itemize}
Thus, up to boundary handling, the leading work per timestep is $\mathcal{O}(N)$ with a constant factor consisting only of integer addition, subtraction and transfer-map evaluation. This constant-factor reduction is only one part of the computational advantage. More importantly, the update has a static, branch-free execution structure: every cell follows the same integer-transfer instruction path, with no nonlinear reconstruction, limiter switching, floating-point control flow, or optimisation loop. Consequently, FQNM is not merely a cheaper stencil computation; it is a stencil whose online execution is almost completely regular.

This regularity enables near-ideal SIMD/SIMT utilisation. More strongly, the update is structurally parallel at the execution level: every cell at a fixed timestep follows an identical branch-free integer-transfer instruction path, with no reconstruction-dependent divergence, limiter switching, floating-point control flow, or optimisation loop. Although the nearest-neighbour data dependence and timestep barrier of an explicit stencil remain, the online computation is a static integer dataflow graph. Conventional high-order shock-capturing schemes, by contrast, often lose hardware efficiency precisely in the regimes where they are needed most: discontinuities and high-frequency structures trigger nonlinear weights, limiters, wider stencils, and irregular control flow. FQNM removes these mechanisms from the online update and replaces them by a fixed integer transfer map. The advantage is therefore not only a smaller arithmetic count, but a fundamentally more parallel execution model.

The formulation therefore aligns naturally with hardware built around discrete operations, and its directional decomposition in higher dimensions is compatible with straightforward parallel execution.

\noindent
The computational comparison is therefore threefold. First, relative to direct floating-point discretisations of the same conservation law, FQNM reduces runtime by eliminating repeated floating-point flux realisation and replacing it with integer-valued transfer execution. This yields lower arithmetic cost and improved hardware utilisation.
  
Second, relative to conventional stencil implementations on parallel hardware, FQNM improves execution regularity. Standard stencil schemes are spatially parallel at a fixed timestep, but their effective parallel utilisation is limited by memory traffic, synchronisation barriers, and, in high-order nonlinear schemes, branch divergence from reconstruction and limiting logic. FQNM does not remove the timestep barrier or the nearest-neighbour communication pattern, but it makes the per-cell computation uniform and branch-free. In this sense, the method enables effectively full parallel utilisation at the execution level, subject only to the usual bandwidth and boundary-exchange constraints of explicit local schemes. Integer arithmetic also avoids floating-point normalisation, rounding, and exception handling, giving lower-latency and more predictable execution on modern hardware.

Third, relative to high-resolution shock-capturing schemes such as WENO, the comparison is not one of formal order but of structural behaviour. In the tested regimes, FQNM maintains less diffusive grid-level shock profiles while incurring significantly lower computational cost, due to the absence of nonlinear reconstruction and floating-point diffusion mechanisms. The important point is the conjunction: the method is not merely faster at the expense of structure, nor more structure-preserving at the expense of speed. Both advantages arise from the same integer-transfer representation.

Thus the reduction in computational cost is not merely empirical but follows from the arithmetic structure of the method. The observed speedups in nonlinear regimes are therefore consistent with the theoretical simplification of the update rule, rather than arising from implementation-specific optimisation.

This also explains why the benefit is expected to be largest in nonlinear conservative regimes. For smooth low-frequency transport, conventional baselines are already close to bandwidth-limited optimal implementations. In contrast, shock-capturing or high-frequency regimes require additional reconstruction, limiting or nonlinear interface logic in floating-point schemes. FQNM removes these layers from the online update, reducing the core arithmetic to a fixed integer transfer map with exact telescoping conservation.

This also clarifies the distinction from PINN-type approaches. While such methods enforce PDE structure through optimisation, they require training and dense matrix operations at inference. FQNM instead preserves the analytic flux explicitly and executes it through a discrete transfer operator, with no learned surrogate and no optimisation loop.

\par

To make the structural reduction explicit, we compare the per-cell arithmetic against a minimal baseline. Even a first-order upwind finite-volume update requires evaluation of the flux function, floating-point multiplications and additions, and accumulation subject to rounding. In contrast, the FQNM update replaces this pipeline by integer transfer-map execution and a fixed number of integer additions/subtractions. Thus, beyond eliminating higher-order reconstruction and limiter logic, the method reduces the arithmetic substrate even relative to the simplest monotone scheme. This confirms that the improvement is not a consequence of removing sophisticated components, but of replacing the floating-point evaluation of the flux operator itself by an integer-valued transfer rule.

\begin{table}[t]
\centering
\caption{\textbf{Representative online arithmetic structure per cell.} The comparison is against a naive first-order upwind finite-volume update, not a high-order reconstruction scheme. Exact cycle counts are architecture-dependent; the point is the structural distinction between floating-point flux evaluation and integer-valued transfer execution.}
\label{tab:alu_structure}
\begin{tabular}{llll}
\toprule
Operation class & Naive upwind FV & FQNM & Consequence \\
\midrule
Flux evaluation & Floating-point \(f(u)\) & Integer transfer map & Removed as FP operation \\
Arithmetic core & FP add/multiply & Integer add/subtract & Lower-latency path \\
Rounding exposure & Per FP operation & Table construction only & No online drift accumulation \\
Branching & Low, but flux-dependent & None in update rule & Static dataflow \\
Conservation & Real-arithmetic telescoping & Integer telescoping & Bitwise exact mass \\
\bottomrule
\end{tabular}
\end{table}

This ALU-level comparison is deliberately made against the weakest conventional baseline. The acceleration reported below is therefore not explained by removing WENO reconstruction, limiter switching, or Riemann-solver complexity. Those mechanisms increase the gap further, but they are not necessary for the gap to appear. The gap already exists between floating-point upwind flux evaluation and integer transfer execution, which is why an order-of-magnitude speedup at the NumPy prototype level is evidence of representation-level arithmetic collapse rather than kernel-specific tuning. This is consistent with the broader observation in accelerator design that memory movement and floating-point realisation often dominate raw arithmetic cost, and that low-precision or integer execution can yield large performance and energy gains when the computational graph is structurally compatible with it \cite{hanEIEEfficientInference2016,buiRealtimeHamiltonjacobiReachability2021,ebrahimiEvaluationFPGAHardware2017}.

Hence runtime per step is
\begin{equation}
T(n) = \mathcal{O}(n).
\label{eq:complexity}
\end{equation}

\subsection{Scope and limitations}
\label{subsec:scope_limitations}

The present formulation is developed and analysed for conservative hyperbolic dynamics, where the evolution can be expressed as local antisymmetric transfer and invariants arise through exact telescoping cancellation. This setting provides the natural starting point for the quantised interaction viewpoint, as conservation laws admit a direct interpretation in terms of countable-state transfer.

This focus should be understood as a choice of setting rather than a fundamental limitation of the representation. The underlying framework operates at the level of transfer rules on a discrete state space, and its applicability is therefore determined by whether a given operator admits a suitable transfer realisation.

Diffusive dynamics, such as the heat equation, can in principle be expressed through symmetric transfer mechanisms within the same framework. Preliminary experiments indicate that the quantised-transfer viewpoint extends beyond purely hyperbolic transport in such cases, although a systematic analysis is not pursued here.

Elliptic constraints, including pressure projection in incompressible flow, do not naturally admit a purely local transfer representation of the type used in the present work, and therefore require a different construction. However, this reflects a difference in operator structure rather than a restriction of the quantised formulation itself: representing such operators within a discrete transfer framework remains a question of transfer-rule design.

The present work therefore does not aim to provide a universal PDE solver. Instead, it establishes that an important class of conservative dynamics can be realised directly as quantised interaction rules, with continuum behaviour emerging as a reconstructed observable. Extensions to broader operator classes are left as future developments of the same representational principle.

\section{Numerical Experiments}
\label{sec:numerics}

We report the simulations analysed in this manuscript. The empirical section focuses on three representative one-dimensional benchmarks: high-frequency Gaussian packet transport for linear advection, inviscid Burgers shock formation, and the Sod shock tube for the compressible Euler system. In the Gaussian transport test, the floating-point comparison baseline is WENO5+RK3, following the high-order WENO constructions and variants in \cite{haModifiedEssentiallyNonoscillatory2013,haSixthorderWeightedEssentially2016,haImprovingAccuracyFifthOrder2021,zhuNewTypeMultiresolution2019,sungTroubledcellIndicatorBased2020}. In the Burgers test, the comparison is made against a standard finite-difference baseline together with a dense-grid entropy reference. In the Sod test, the comparison is made under the same Roe-type approximate Riemann flux structure, isolating the effect of quantised conservative transfer at the Euler-system level.

These experiments are intended as focused validation in complementary regimes---oscillatory transport near the grid Nyquist limit, nonlinear scalar shock formation in Burgers dynamics, and coupled system-level shock propagation in the Euler equations---rather than as a broad benchmark survey. They deliberately separate high-frequency content from nonlinear system-level shock interaction: the Gaussian packet test isolates oscillatory transport, the Burgers test isolates scalar shock formation, and the Sod test isolates coupled Euler shock propagation under matched Roe-flux conditions. We therefore do not claim here to resolve an oscillatory nonlinear system in a single benchmark; such coupled shock--oscillation tests remain a natural next stress case. The purpose of the present set is to isolate the operator-quantisation effect across the main ingredients separately.

We also report wall-clock timings from the current prototype in nonlinear shock-dominated regimes. These timings are not used as the sole evidence for the method, but they are important because they test whether the arithmetic simplification predicted in \cref{subsec:complexity} appears in actual execution. In Burgers-type nonlinear benchmarks, the integer-transfer implementation achieves approximately \(10\text{--}13\times\) acceleration over a floating-point finite-difference baseline at matched resolution, while preserving the grid-level shock structure and exact discrete conservation. This places the speedup in precisely the regime where conventional floating-point schemes incur their largest reconstruction, stencil, and control-flow overheads.

Importantly, this acceleration is not attributable to low-level kernel optimisation. The comparison is made at the level of high-level array execution, and the arithmetic reduction already appears when the reference is a naive first-order upwind finite-volume update rather than a high-order shock-capturing method. Direct ALU-level inspection shows that the classical upwind update still requires floating-point flux evaluation, arithmetic combination, and rounding-sensitive accumulation, whereas FQNM reduces the online per-cell update to integer transfer-map evaluation and antisymmetric transfer. This establishes that the observed speedup is structural: it arises from eliminating floating-point flux realisation and collapsing the arithmetic pipeline, rather than from hardware-specific tuning. The fact that an order-of-magnitude acceleration already appears in a high-level NumPy implementation therefore provides a lower bound on the achievable performance gain under the change of representation.

\par
It is important to emphasise that the present experiments do not rely on precomputed lookup tables. The key distinction is the quantisation of the flux operator itself, not the use of table-based execution.

In the floating-point formulation, the operator is realised through repeated evaluation of the continuum flux on reconstructed field values. In the quantised formulation, the same operator is executed as an integer-valued transfer rule acting on discrete states. Thus the comparison isolates the effect of arithmetic realisation rather than introducing a different numerical scheme.

Lookup-based execution is a possible optimisation enabled by the discrete state space, but it is not required for the structural behaviour observed here. The less diffusive shock localisation and the reduction in arithmetic complexity arise from the quantised transfer representation itself.

\subsection{High-frequency transport stress test}
\label{subsec:hf_stress}

To probe dispersion and diffusion failure under severe resolution constraints, we advect a high-frequency Gaussian wave packet and compare three final-time profiles: the exact shift solution, FQNM, and a floating-point high-order baseline (WENO5+RK3). We sweep the normalised frequency $k_0\Delta x$ and report the relative $L^2$ error against the exact solution. For each value of $k_0\Delta x$, we aggregate multiple $(n_x,f)$ settings that realise the same normalised frequency and summarise the resulting error distribution by its median together with a min--max envelope.

The point of this test is structural rather than merely empirical. Near the Nyquist limit, classical high-order reconstructions become increasingly sensitive to dispersion and phase error, whereas the FQNM update acts directly through local conservative transfer on quantised states. The comparison therefore probes whether that interaction-rule formulation remains stable in a regime where floating-point field reconstructions visibly deteriorate.

\begin{figure}[t]
\centering
\IfFileExists{figs/bandplot_kdx.png}{%
  \includegraphics[width=0.92\linewidth]{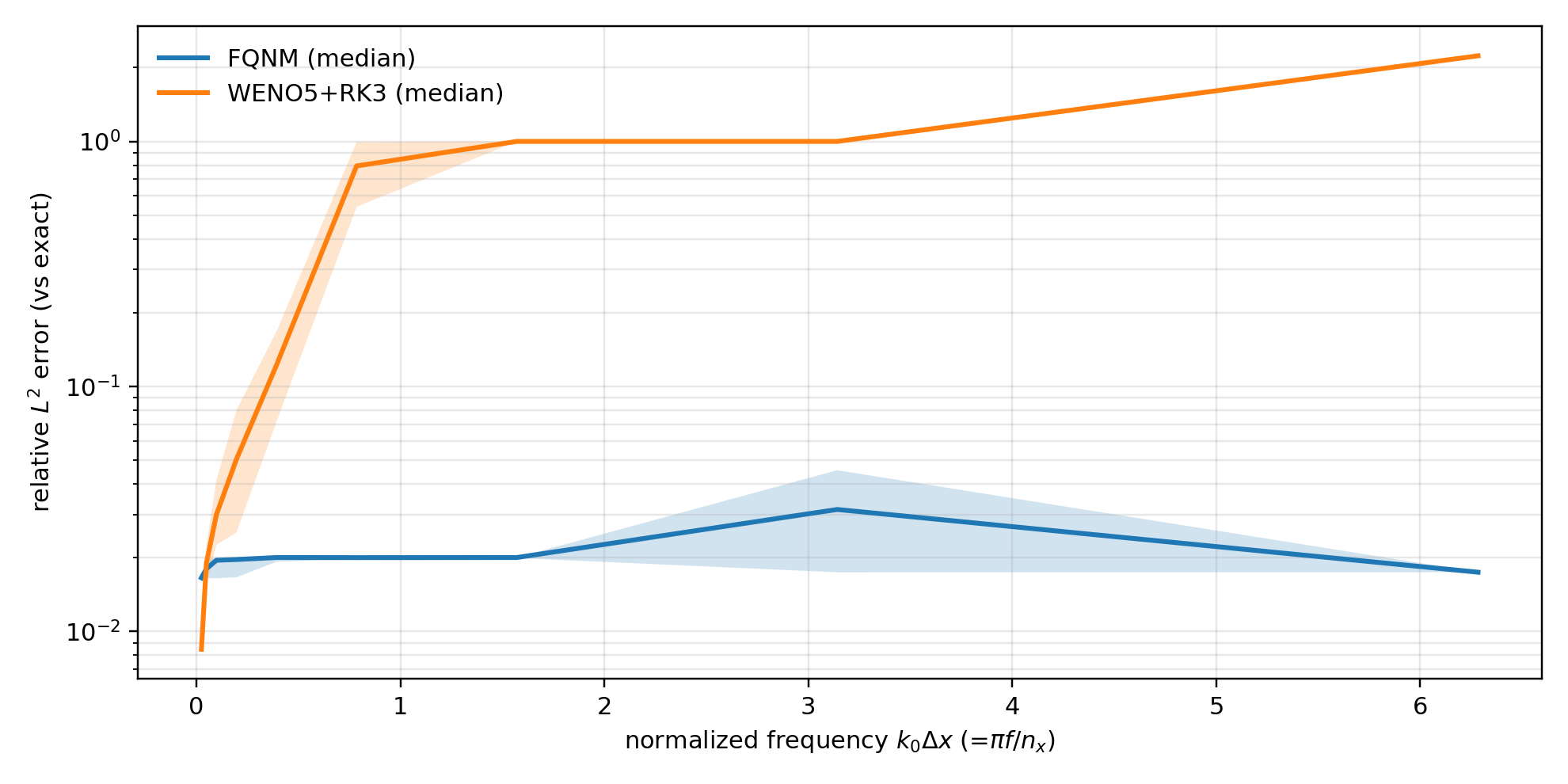}%
}{%
  \fbox{\parbox{0.88\linewidth}{\centering Placeholder: `figs/bandplot_kdx.png` not found at compile time. Restore the Nyquist-regime band-summary figure to this path.}}%
}
\caption{\textbf{High-frequency transport band summary.} Relative $L^2$ error as a function of normalised frequency $k_0\Delta x$. FQNM remains accurate as the Nyquist regime is approached, whereas the floating-point WENO5+RK3 baseline shows rapidly increasing error. Bands show the min--max envelope and the central line shows the median across matched $(n_x,f)$ settings.}
\label{fig:bandplot_kdx}
\end{figure}

\subsection{Nonlinear Burgers shock regime}
\label{subsec:burgers}

To assess nonlinear robustness beyond linear transport, we consider the inviscid Burgers equation and summarise the comparison with a single representative snapshot in \cref{fig:burgers_shock_zoom}.
\[
\partial_t u + \partial_x \left( \frac{u^2}{2} \right)=0
\]
with smooth initial data
\[
 u_0(x)=0.5+\sin(2\pi x),
\]
evolved past the shock formation time.

We compare FQNM against a standard finite-difference baseline using a representative shock snapshot, and also against a dense-grid entropy reference through separate diagnostic evaluation. This Burgers test serves as the nonlinear one-dimensional benchmark in the paper.

In the same nonlinear regime, the prototype integer-transfer implementation shows a wall-clock speedup of roughly \(11\times\) over the floating-point finite-difference baseline. Earlier benchmark sweeps on Burgers and related one-dimensional viscous transport tests gave representative speedups of approximately \(10.9\times\) and \(11.1\times\), respectively. We use these values only as prototype-level timing evidence; the central point is that the largest acceleration appears in the shock-forming nonlinear regime where the structural simplification of FQNM is expected to matter most.

\begin{figure}[t]
\centering
\resizebox{0.9\linewidth}{!}{%
  \IfFileExists{figs/burgers_shock_zoom.pgf}{%
    \input{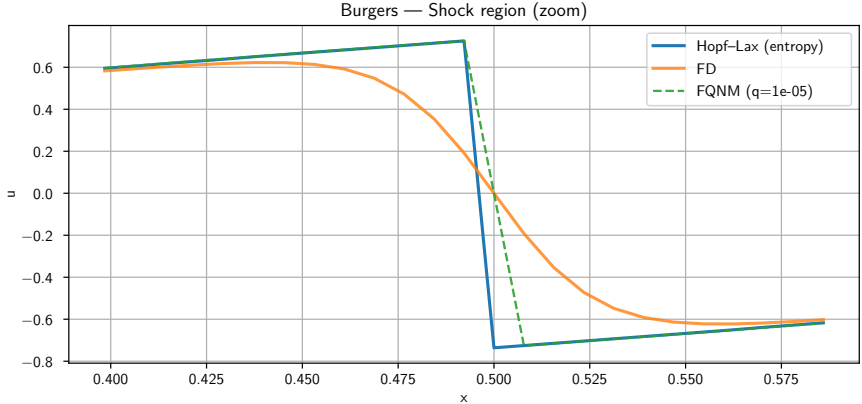}%
  }{%
    \fbox{\parbox{0.82\linewidth}{\centering Placeholder: `figs/burgers_shock_zoom.pgf` not found at compile time. Rebuild the figure file or restore it to the `figs/` directory.}}%
  }%
}
\caption{\textbf{Burgers shock comparison (single snapshot).} Representative Burgers shock profile comparison between a standard finite-difference baseline and FQNM at a matched resolution and final time. Relative to the finite-difference baseline, FQNM preserves shock location while exhibiting reduced numerical diffusion. The offset initial data induce a pinned grid-level structure, and the comparison shows that FQNM is more robust to one-cell shock displacement (cell drifting) while maintaining exact discrete conservation.}
\label{fig:burgers_shock_zoom}
\end{figure}

For the inviscid Burgers case, we additionally compare both numerical schemes against a Hopf--Lax entropy reference evaluated on a dense auxiliary grid. The Hopf--Lax minimisation uses $N_y = 2048$ candidate points in $y$, whereas all reported profiles are sampled and compared on the simulation grid with $N_x = 128$. This dense-grid entropy reference is used to separate continuum entropy admissibility from grid-level placement effects in the discrete profiles.

A subtle but important point concerns the offset structure of the initial data. The constant shift fixes a reference level that should remain pinned under the Burgers evolution, including after shock formation. In the discrete simulations, FQNM preserves this pinned grid-level structure through the integer state evolution itself, whereas reconstruction-based references and floating-point baselines can exhibit one-cell placement ambiguity near the shock. This does not indicate a failure of the entropy solution; it is a sampling artefact of projecting the continuous minimisation onto an even grid, where a shock located exactly between two cells must be assigned to one adjacent grid point. Here, `cell drifting' refers to a one-cell displacement of the sampled shock relative to that pinned grid-level offset structure. Against this backdrop, the comparison is not a claim that FQNM outperforms the entropy solution itself. The point is sharper: FQNM preserves a discrete structural invariant of the quantised evolution while simultaneously delivering the observed wall-clock acceleration, a combination that is not available to floating-point reconstruction-based updates in this regime.

In addition to conservation metrics, we measured the discrete entropy defined in \cref{eq:discrete_entropy} during the inviscid Burgers simulations. After shock formation, the empirical entropy $S^n$ increases in regimes where the induced level-transition matrix is close to bistochastic. In the present mean-field implementation (\cref{eq:mean_field_closure}), this behaviour should be interpreted as deterministic redistribution in a countable-state conservative scheme, rather than noise-induced stochastic stabilisation.

\subsection{System-level Sod shock tube under Roe flux quantisation}
\label{subsec:sod_system_level}

We next test whether the same quantised-transfer principle remains effective for a coupled hyperbolic system rather than a scalar conservation law. For this purpose, we use the one-dimensional Sod shock tube for the compressible Euler equations. Both trajectories are run under the same initial data, grid, CFL condition, final time, and Roe-type approximate Riemann flux structure. The distinction is not a change of Riemann solver: both solvers use the same Roe flux physics. The difference is in the execution layer. The reference trajectory evolves floating-point conservative variables through a Roe flux-difference update, whereas the FQNM trajectory converts the same Roe flux into an integer-valued conservative transfer update on quantised conservative states.

In the full system formulation, the quantised-transfer construction can be applied componentwise to all conservative variables \((\rho,m,E)\). In the prototype reported here, however, we isolate the density channel: operator quantisation is imposed at the density-transfer level, while velocity and pressure are reconstructed from the coupled conservative state rather than treated as independently quantised primitive variables.

Thus this experiment does not test the strongest possible system quantisation, in which density, momentum, and energy transfers are all quantised simultaneously. Instead, it asks a more restrictive question: whether quantising the density transfer operator alone can reduce the effective artificial-viscosity signature seen in the shock structure. Velocity and pressure remain reconstructed observables and therefore still inherit the diffusive characteristics of the underlying Roe flux.

Despite this restricted setup, the experiment shows that density-channel operator quantisation alone is already sufficient to improve the structural localisation of the shock and to reduce the visible artificial-viscosity effect in the density profile.

This experiment should therefore be read as a system-level arithmetic-realisation test with an exact Riemann reference included for calibration. It does not introduce a different physical flux. Rather, the same Roe flux structure is realised through quantised transfer.

Because velocity and pressure are reconstructed rather than independently quantised, they remain subject to the diffusive and artificial-viscosity characteristics of the Roe flux. In particular, no claim is made that numerical diffusion is uniformly reduced across all variables or that the result exhausts the behaviour of a fully componentwise-quantised Euler system.

The key observation is instead structural: even density-channel operator quantisation alone can preserve and sharpen the shock transition at the conserved-density level. The improvement should therefore be interpreted as a consequence of how the density transfer is realised, rather than as a uniform accuracy gain in the full primitive-variable system.

\begin{figure}[t]
\centering
\resizebox{0.92\linewidth}{!}{%
  \IfFileExists{figs/sod_fd_vs_fqnm_full.pgf}{%
    \input{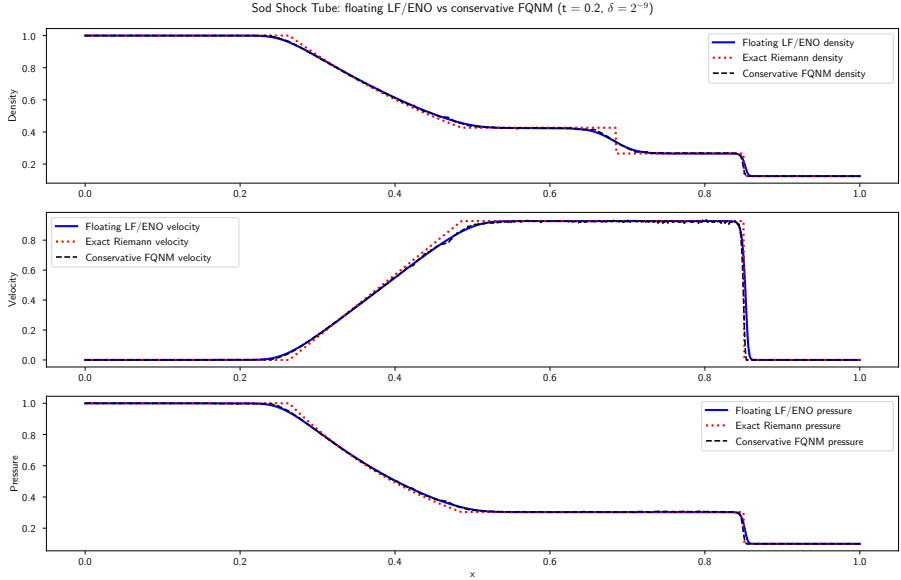}%
  }{%
    \fbox{\parbox{0.86\linewidth}{\centering Placeholder: `figs/sod_fd_vs_fqnm_full.pgf` not found at compile time. Restore the full-domain Sod FD--FQNM comparison PGF to this path.}}%
  }%
}
\caption{\textbf{System-level Sod shock tube: full-domain density profile.} Floating-point Roe reference and FQNM are compared under identical initial data, grid, CFL condition, final time, and Roe-type approximate Riemann flux structure. The difference lies solely in the execution layer: the reference evolves floating-point conservative variables, whereas FQNM realises the same Roe-induced transfer through a quantised conservative update. The exact Sod density profile is shown as a calibration reference. In this prototype setting, quantising only the density transfer operator already preserves and sharpens the density shock structure, even though velocity and pressure are reconstructed quantities and remain subject to the diffusive characteristics of the Roe flux.}
\label{fig:sod_fd_fqnm_full}
\end{figure}

\begin{figure}[t]
\centering
\resizebox{0.92\linewidth}{!}{%
  \IfFileExists{figs/sod_fd_vs_fqnm_zoom.pgf}{%
    \input{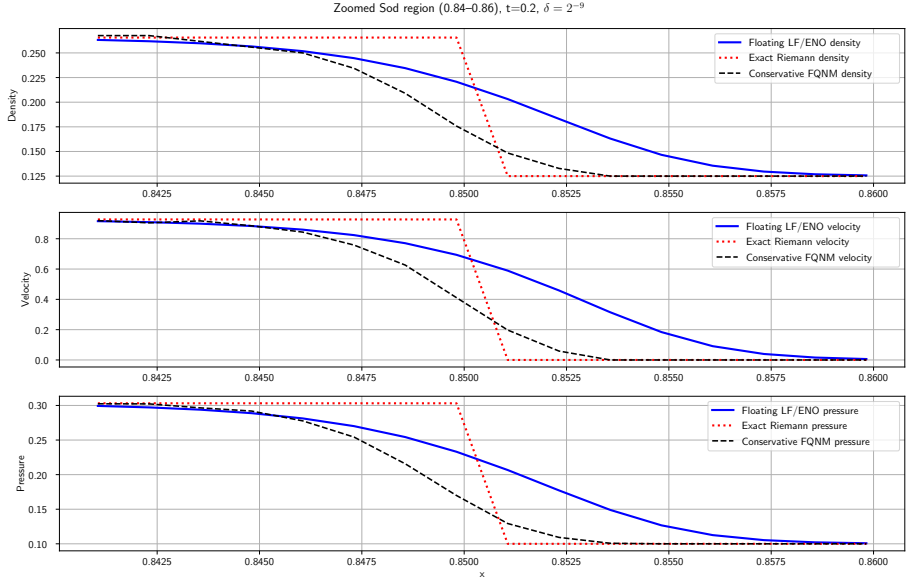}%
  }{%
    \fbox{\parbox{0.86\linewidth}{\centering Placeholder: `figs/sod_fd_vs_fqnm_zoom.pgf` not found at compile time. Restore the zoomed Sod FD--FQNM comparison PGF to this path.}}%
  }%
}
\caption{\textbf{Zoomed Sod shock region.} Magnification of the density profile over the interval $x\in[0.84,0.86]$. This view isolates the shock transition under identical Roe flux structure. The comparison highlights that quantising the density transfer operator alone reduces the visible artificial-viscosity effect in the density shock profile, while velocity and pressure remain reconstructed variables influenced by Roe-type numerical diffusion. The figure is intended to illustrate structural behaviour rather than a uniform accuracy improvement across all variables.}
\label{fig:sod_fd_fqnm_zoom}
\end{figure}

\section{Discussion}
\label{sec:discussion}

The central point of FQNM is representational. Classical conservation-law solvers evolve reconstructed floating-point fields and enforce structure through additional layers: conservative flux form, Riemann solvers, nonlinear reconstruction, limiters, or optimisation. FQNM instead executes the conservative transfer itself as the primitive operation. The continuum field is recovered only after the integer transfer dynamics have been applied.

This distinction is not merely semantic. Exact conservation in FQNM is an invariant of the arithmetic substrate. In a reconstruction-based floating-point method, conservation holds at the level of the ideal real-valued flux difference, while the realised computation is subject to finite-precision cancellation. In FQNM, the update is an integer-valued antisymmetric transfer, so global conservation follows by exact telescoping in the computation actually performed. This is the one-dimensional boundary-free analogue of Green/Stokes cancellation in discrete transfer form.

The same representation also explains the computational reduction. The online update does not evaluate a floating-point flux on reconstructed states. It applies a fixed integer transfer map. Consequently, the arithmetic graph of the update collapses from functional evaluation and reconstruction to nearest-neighbour integer transfer. This removes floating-point flux evaluation, limiter logic, nonlinear control flow, and reconstruction-dependent branching from the executed update. The resulting acceleration is therefore not a kernel optimisation of a conventional stencil, but a consequence of changing the realised operator.

This point is visible even against the weakest conventional baseline. A naive first-order upwind finite-volume update still evaluates a floating-point flux, combines floating-point quantities, and accumulates rounding-sensitive updates. FQNM replaces this arithmetic substrate by integer transfer-map execution and exact antisymmetric cancellation. Thus the observed order-of-magnitude prototype speedup in Burgers-type nonlinear regimes should not be interpreted as merely avoiding WENO reconstruction or high-order machinery; the arithmetic gap already appears before those higher-order layers are introduced.

The shock behaviour should be read in the same way. FQNM is not a more finely tuned shock-capturing scheme. In floating-point reconstruction-based methods, shocks are represented through steep reconstructed gradients and stabilised through Riemann, limiting, or entropy mechanisms. In FQNM, discontinuities appear as abrupt changes in the local transfer pattern. Shock localisation, exact conservation, and branch-free execution are therefore not separate algorithmic devices; they are consequences of the same integer transfer representation.

The Burgers experiment demonstrates this scalar mechanism. The offset initial data induce a pinned grid-level structure, and FQNM preserves this discrete structural invariant while also showing the expected wall-clock acceleration. This should not be interpreted as outperforming the continuum entropy solution. Rather, it shows that the quantised transfer dynamics preserve a structural feature of the discrete state evolution that is not preserved in the same way by reconstruction-based floating-point updates.
The Sod shock-tube comparison provides a prototype system-level test of the same principle. Both trajectories use the same initial data, grid, CFL condition, final time, and Roe-type approximate Riemann flux structure, and the exact Sod solution is included as a calibration reference. The distinction is not the physical flux law or the Riemann solver, but the execution layer: the reference trajectory evolves floating-point conservative variables through a Roe flux-difference update, whereas the FQNM trajectory realises the same Roe-induced transfer as an integer-valued conservative update.

In this implementation, the full conservative system could in principle be quantised componentwise, but the reported prototype isolates the density-transfer channel. Velocity is reconstructed from the coupled conservative state rather than quantised as an independent primitive variable.

As a result, velocity and pressure remain subject to the numerical diffusion and artificial-viscosity mechanisms of the underlying Roe flux. The observed improvement in shock structure therefore does not arise from fully quantising all primitive variables.

Instead, the key point is that quantisation of the density transfer operator alone is sufficient to stabilise and sharpen the density shock profile, reducing the visible artificial-viscosity effect even before full componentwise system quantisation is used. This indicates that a substantial part of the structural behaviour of the shock is controlled by the discrete transfer of conserved density, rather than by the detailed floating-point reconstruction of derived variables.

The transfer-operator equivalence theorem makes this viewpoint precise. Once two pre-quantisation flux formulae induce the same integer interface transfer map on the states visited by the computation, their FQNM trajectories are identical. The effective computational object is therefore not the analytic flux expression prior to quantisation, but the realised transfer operator. This collapse of classical flux distinctions is one of the clearest signatures that FQNM is not simply a low-precision finite-volume implementation.

The entropy construction is treated separately in Supplementary Information as an interpretive statistical diagnostic. Conservation fixes the total integer mass, while nonlinear transfer redistributes occupation over quantised levels. The empirical level-transition matrix describes mixing in level space, not permutation of grid points. When this induced level mixing is close to bistochastic, the discrete Shannon entropy increases by majorisation. This appendix-level interpretation is not used in the consistency, stability, or convergence arguments above.

Taken together, the results support four main conclusions. First, FQNM makes conservation an exact arithmetic invariant through antisymmetric integer transfer. Second, it changes the computational primitive from field reconstruction to operator transfer. Third, this change collapses the online arithmetic graph to a static branch-free integer update, explaining the observed prototype speedups. Fourth, the same transfer representation preserves shock-relevant discrete structure in scalar tests and remains effective in a matched Roe-flux Sod prototype.

The limitations are equally important. The present theory is established for scalar hyperbolic conservation laws with monotone split fluxes. The Sod test is a prototype system-level arithmetic-realisation experiment, not a complete convergence theory for compressible Euler. Diffusive, elliptic, source-driven, and multiphysics systems require transfer representations adapted to their operator structure. The quantisation scale \(\delta\) also introduces a practical trade-off between resolution, state-space size, and memory footprint.

Thus the main implication is not that FQNM is a universal PDE solver. It is that, for conservative hyperbolic dynamics, the object chosen for computation can be changed. Replacing reconstructed field evolution by direct execution of the conservative transfer operator can alter both structural fidelity and arithmetic cost at the same time.

\paragraph{Practical applicability and limitations.}

The present formulation is most directly applicable to conservative hyperbolic problems in which the dynamics can be expressed as local flux transfer. This includes compressible flow, nonlinear wave propagation, and transport-dominated regimes where shock formation and high-frequency structure dominate computational cost. In such settings, the integer-transfer realisation can replace floating-point flux evaluation without modifying the governing law, offering a direct path to reduced arithmetic cost and improved hardware efficiency.

From an implementation perspective, the method is compatible with standard finite-volume infrastructures: the required modifications are confined to the flux evaluation layer, which is replaced by a quantised transfer rule. This suggests that FQNM can be incorporated into existing solvers without redesigning grid structures, boundary handling, or time integration frameworks.

However, the applicability is currently limited by the availability of suitable transfer representations. Diffusive operators require symmetric transfer constructions, and elliptic constraints introduce nonlocal coupling that is not captured by the present local interaction rule. In addition, the choice of quantisation scale $\delta$ introduces a trade-off between resolution and state-space size, which may affect memory footprint and table construction in high-dimensional or multi-physics settings.

Therefore, while the results demonstrate that operator-level quantisation can provide substantial advantages in shock-dominated conservative regimes, extending this approach to general PDE systems remains an open problem requiring further development of transfer-based operator representations.

\section{Conclusion}
\label{sec:conclusion}

This work introduces a representation in which conservative hyperbolic dynamics are realised as integer-valued transfer on a countable state space, with continuum fields appearing only as reconstructed observables. The resulting scheme preserves conservation exactly at the arithmetic level and reduces the computational structure of the update to a fixed transfer rule.

Across representative regimes, this representation maintains structural fidelity while reducing computational complexity. In particular, under matched Roe-flux conditions in the Sod shock-tube system test, the quantised transfer realisation produces a less diffusive shock transition and reduced density error relative to the exact reference within the tested prototype setting.

These results should be interpreted as evidence for a system-level effect of operator realisation rather than as a general claim of convergence or superiority for fully quantised multi-variable systems. The present work establishes that, even under a single density-scale conservative-state quantisation, operator-level structure can be preserved in nonlinear and system-level dynamics.

This suggests that, for conservative hyperbolic problems, the choice of operator representation can be as important as the choice of numerical flux or reconstruction, and that quantised transfer provides a viable alternative execution paradigm for such systems.

\bibliography{sn-bibliography.used}
\newpage
\section*{Supplementary Information}
\label{sec:si}
\setcounter{theorem}{0}
\renewcommand{\thetheorem}{S\arabic{theorem}}
\subsection*{Entropy and Statistical Structure}
\label{app:entropy}
The entropy construction in this appendix is used only as an interpretive statistical diagnostic for quantised transfer dynamics. It is not part of the consistency, stability, or convergence theory in the main text.

\subsection*{Discrete Entropy Functional}
\label{subsec:entropy_functional}

Let
\begin{equation}
p_k = \frac{1}{N} \#\{ i \mid q_i = k \},
\label{eq:pk_definition}
\end{equation}
where $N$ is the number of grid points. Define
\begin{equation}
S = - \sum_k p_k \log p_k.
\label{eq:discrete_entropy}
\end{equation}

The entropy \cref{eq:discrete_entropy} is the Shannon entropy of the empirical level distribution induced by the quantised state configuration. Because the local states are explicitly countable, this provides a natural coarse-grained measure of level mixing in the conservative dynamics \cite{shannonMathematicalTheoryCommunication1948,jaynesInformationTheoryStatistical1957,coverElementsInformationTheory2006}.

\begin{proposition}[Log-count structure: Boltzmann and Shannon as the same principle]
\label{prop:log_count}
Let $p$ be the empirical level distribution \cref{eq:pk_definition} on a finite set of occupied levels.
(i) If $p$ is uniform on $m$ occupied levels, then $S=\log m$.
(ii) More generally, define the \emph{effective number of occupied levels} by
\begin{equation}
N_{\mathrm{eff}}(p):=\exp(S(p)).
\label{eq:neff}
\end{equation}
Then the Shannon entropy can be written as the logarithm of an effective count: $S=\log N_{\mathrm{eff}}$.
In particular, when the microscopic state space is uniformly accessible on a finite set $\Omega$ of size $|\Omega|$, the Boltzmann form $S_B = k\log|\Omega|$ and the coarse-grained Shannon form $S= -\sum_k p_k\log p_k$ are the same log-count principle applied at different resolutions.
\end{proposition}

\begin{proof}
If $p_k=1/m$ on $m$ occupied levels, then $S=-\sum_{k=1}^m (1/m)\log(1/m)=\log m$.
The definition \cref{eq:neff} gives $S=\log N_{\mathrm{eff}}$ by construction.
\end{proof}

\subsection*{Entropy Growth under Nonlinear Redistribution}
\label{subsec:entropy_growth}

While \cref{eq:exact_conservation} guarantees invariant total mass, nonlinear fluxes redistribute occupation numbers across levels.

\begin{definition}[Empirical level-transition matrix]
\label{def:level_transition}
Let $c_k^n := \#\{ i \mid q_i^n = k \}$ and $p_k^n=c_k^n/N$ as in \cref{eq:pk_definition}. Define the one-step transition counts
\[
N_{k'k}^n := \#\{ i \mid q_i^n = k\ \text{and}\ q_i^{n+1}=k' \},
\]
so that $\sum_{k'} N_{k'k}^n = c_k^n$ and $\sum_k N_{k'k}^n = c_{k'}^{n+1}$. For $c_k^n>0$, define
\begin{equation}
M_{k'k}^n := \frac{N_{k'k}^n}{c_k^n},
\label{eq:empirical_M}
\end{equation}
and set $M_{k'k}^n:=0$ when $c_k^n=0$. Then $M^n$ is column-stochastic on the occupied levels, in the sense that
\[
\sum_{k'} M_{k'k}^n = 1 \qquad \text{whenever } c_k^n>0,
\]
and satisfies
\begin{equation}
p^{n+1} = M^n p^n.
\label{eq:pn1_equals_Mpn}
\end{equation}
\end{definition}

\begin{remark}[Interpretation of $M^n$]
\label{rem:M_interpretation}
The matrix $M^n$ does not represent a permutation of grid indices. 
Instead it describes redistribution between quantised \emph{levels} induced by the conservative flux update.
Because the local antisymmetric transfer generally changes the multiplicity of each level, the resulting level-transition operator is typically not a permutation matrix. 
Thus the entropy discussion concerns mixing in level space rather than spatial rearrangement of the grid state.
\end{remark}

\begin{remark}[When is $M^n$ doubly stochastic?]
\label{rem:doubly_stochastic_condition}
By construction, $M^n$ satisfies $\sum_{k'} M_{k'k}^n=1$ on the occupied levels, i.e. whenever $c_k^n>0$. It is \emph{doubly} stochastic precisely when, in addition, $\sum_k M_{k'k}^n=1$ for every $k'$, i.e. when the induced level-transition mechanism has no net preference for any output level at the coarse-grained scale.
In practice we quantify deviations from bistochasticity by the row-sum defect
\(\rho^n := \max_{k'}\big|\sum_k M_{k'k}^n-1\big|\).
When \(\rho^n\) is small over the occupied levels, \cref{thm:entropy_increase} provides a sharp idealisation of the mixing mechanism, and we interpret observed entropy growth as operating in a near-bistochastic regime.
We do not claim that FQNM enforces bistochasticity in general; instead, bistochasticity is treated as an empirically testable condition quantified by \(\rho^n\).
\end{remark}

\begin{theorem}[Discrete entropy increase under bistochastic level mixing]
\label{thm:entropy_increase}
Let $p^n$ be defined by \cref{eq:pk_definition}, and let $M^n$ be the empirical level-transition matrix from \cref{def:level_transition}. If $M^n$ is doubly stochastic, then the discrete entropy \cref{eq:discrete_entropy} satisfies
\begin{equation}
S^{n+1} \ge S^n.
\label{eq:entropy_monotone}
\end{equation}
If, in addition, $p^{n+1}=M^n p^n$ is not a permutation of $p^n$, then the inequality is strict:
\begin{equation}
S^{n+1} > S^n.
\label{eq:entropy_strict}
\end{equation}
\end{theorem}

\begin{proof}
Restrict $M^n$ to the finite union of the supports of $p^n$ and $p^{n+1}$.
By \cref{eq:pn1_equals_Mpn}, we then have $p^{n+1}=M^n p^n$ on this finite-dimensional state space.
If $M^n$ is doubly stochastic, then $p^{n+1}$ is majorized by $p^n$ (Hardy--Littlewood--P\'olya; see \cite{marshallInequalitiesTheoryMajorization2011}).
Shannon entropy $S(p)=-\sum_k p_k\log p_k$ is Schur-concave, hence $S(p^{n+1})\ge S(p^n)$.
If, moreover, $p^{n+1}$ is not a permutation of $p^n$, then strict Schur-concavity gives $S(p^{n+1})>S(p^n)$.
\end{proof}

\begin{definition}[Discrete entropy production rate]
\label{def:entropy_rate}
Define the one-step entropy production rate by
\begin{equation}
\dot S^{n+\frac12} := \frac{S^{n+1}-S^n}{\Delta t}.
\label{eq:entropy_rate}
\end{equation}
\end{definition}

\subsection*{Continuous Limit}
\label{subsec:continuous_limit}

Under the consistency result of \cref{thm:consistency}, as $\delta \to 0$ with $\nu \le 1$ and $\delta/\Delta x \to 0$, the reconstruction $u_i = (\mathcal R_\delta q)_i$ converges (along subsequences) to a weak solution of \cref{eq:conservation_law} in the standard Lax framework for hyperbolic conservation laws \cite{evansPartialDifferentialEquations2010,toroRiemannSolversNumerical2009}.

As $\delta \to 0$, the reconstruction approaches a continuum field and \cref{eq:discrete_entropy} formally approaches continuum entropy functionals used in conservation-law theory. In this sense, macroscopic entropy inequalities emerge as scaling limits of combinatorial redistribution on a countable lattice.

\end{document}